\documentclass[11pt]{article}
\usepackage{amssymb,amsmath,amsthm}
\usepackage{enumerate}
\usepackage{pgf,tikz}
\usetikzlibrary{arrows}

\newtheorem{theorem}{Theorem}[section]
\newtheorem{corollary}[theorem]{Corollary}
\newtheorem{lemma}[theorem]{Lemma}
\newtheorem{proposition}[theorem]{Proposition}

\newtheorem{example}[theorem]{Example}
\numberwithin{equation}{section}

\newcommand{\Pf}{\operatorname{Pf}}

\newcommand{\wedgedots}{\wedge \cdots \wedge}

\newcommand{\Nnn}{\mathbb{N}}

\begin{document}

\author{Richard Ehrenborg and N. Bradley Fox}
\title{A sign-reversing involution for an extension of Torelli's Pfaffian identity}
\date{}

\maketitle

\begin{abstract}
We evaluate the hyperpfaffian
of a skew-symmetric 
$k$-ary polynomial $f$
of degree $k/2 \cdot (n-1)$.
The result is a product of the Vandermonde product and
a certain expression involving the coefficients of the polynomial $f$.
The proof utilizes a sign reversing involution on a set of weighted, oriented partitions.
When restricting to the classical case when $k=2$
and the polynomial is $(x_{j} - x_{i})^{n-1}$,
we obtain an identity due to Torelli.
\end{abstract}

\section{Introduction}

The Pfaffian of a skew-symmetric matrix
is commonly defined as the square root of the determinant.
Note that if the order of the matrix is odd, then the determinant
vanishes and the Pfaffian is zero.
Hence we assume that the order is even.
Similar to the determinant (of any square matrix) being expressed
as a sum over all perfect matchings of the complete bipartite
graph, the Pfaffian has an explicit expression as a sum over all
perfect matchings of the complete graph.

Barvinok~\cite{Barvinok} extended the notion of
the Pfaffian to the hyperpfaffian.
Instead of considering matchings of the complete graph,
consider set partitions of the set $[n] = \{1,2, \ldots, n\}$
into blocks of equal size $k$. 
Let $\Pi_{n,k}$ denote the set of such set partitions.
Furthermore, let $k$ be an even integer and $n$ a multiple of~$k$.
Let $f$ be a $k$-ary skewsymmetric function defined on
the set $[n]^{k}$.
For a $k$-element subset $B =\{b_{1} < b_{2} < \cdots < b_{k}\}$
of $[n]$ write $f(B) = f(b_{1}, b_{2}, \ldots, b_{k})$.
Lastly, define the sign $(-1)^{\tau}$
of a partition $\tau = \{B_{1}, B_{2}, \ldots, B_{n/k}\}$
in $\Pi_{n,k}$ to be the sign of the permutation
$b_{1,1}, b_{1,2}, \ldots, b_{1,k}, b_{2,1}, \ldots, b_{2,k},
b_{3,1}, \ldots, b_{n/k,k}$, where the $i$th block $B_{i}$ is given
by $B_{i} = \{b_{i,1} < b_{i,2} < \cdots < b_{i,k}\}$.
Then the hyperpfaffian is defined 
by the sum
\begin{equation}
     \Pf(f)
    =
      \sum_{\tau}
                  (-1)^{\tau}
                \cdot
                   \prod_{i=1}^{n/k} f(B_{i})    ,
\label{equation_Pfaffian}
\end{equation}
where the sum is over all partitions
$\tau=\{B_{1},B_{2},\ldots, B_{n/k}\}$
in $\Pi_{n,k}$;
see~\cite[Section~3]{Barvinok}.

In the case when the function is a skew-symmetric polynomial $f$
in $k$ variables of degree $k/2 \cdot (n-1)$,
we evaluate the hyperpfaffian; see
Theorem~\ref{main_theorem}.
The result is a product of the Vandermonde product
and an expression of the coefficients of the polynomial~$f$.
We prove this using a sign reversing involution
that cancels most of the terms, leaving only the terms corresponding
to the Vandermonde determinant.
The proof can be made completely combinatorial by combining
the last step
with Ira Gessel's sign reversing involution in
his proof of the Vandermonde identity~\cite{Gessel}.
In the classical Pfaffian case, that is, when $k=2$,
our identity yields a nice expression,
generalizing an identity due to Torelli~\cite{Torelli}.

In the last section we state some open questions
about the hyperpfaffian, among them what identities
does it satisfy.

\section{The hyperpfaffian in connection with the exterior algebra}

To give more motivation for the hyperpfaffian
we introduce the exterior algebra.
Recall that $f$ is a skew-symmetric function if
for all permutations $\sigma$ in ${\mathfrak S}_{k}$ we have that
$$ f(i_{\sigma(1)}, i_{\sigma(2)}, \ldots, i_{\sigma(k)})
      =
          (-1)^{\sigma}
      \cdot
           f(i_{1}, i_{2}, \ldots, i_{k}) , $$
where $(-1)^{\sigma}$ denotes the sign of the permutation $\sigma$.
Observe that if two of the entries $i_{1}, i_{2}, \ldots, i_{k}$
are equal, then the function value
$f(i_{1}, i_{2}, \ldots, i_{k})$ is equal to zero.

Let $\Lambda$ denote the exterior algebra
in the variables $t_{1}, t_{2}, \ldots, t_{n}$.
For $S = \{s_{1} < s_{2} < \cdots < s_{m}\}$
a subset of $[n]$, let $t_{S}$
denote the exterior product
$t_{S} = t_{s_{1}} \wedge t_{s_{2}} \wedgedots t_{s_{m}}$.
Observe that
for two sets $S$ and $T$ that share at least one element,
we have that $t_{S} \wedge t_{T} = 0$.
Also note that if at least one of the two sets $S$ and $T$
have even cardinality, then the elements $t_{S}$ and $t_{T}$
commute, that is,
$t_{S} \wedge t_{T} = t_{T} \wedge t_{S}$.
Furthermore, let $f(S)$ denote
the function value $f(s_{1}, s_{2}, \ldots, s_{k})$.

Luque and Thibon expressed the hyperpfaffian
in terms of the exterior algebra~\cite[Equation~(79)]{Luque_Thibon}.
We include a proof for completeness.
\begin{proposition}
The {\em hyperpfaffian} of the skew-symmetric function
$f$ defined on the set $[n]^{k}$ is the unique scalar
given by the equation
$$
\left(
\sum_{S}
    f(S) \cdot t_{S}
\right)^{n/k}
=
(n/k)!
\cdot
   \Pf(f) \cdot t_{[n]}  ,
$$
where the sum is over all $k$-element
subsets of the set $[n]$.
\label{proposition_Pfaffian}
\end{proposition}
\begin{proof}
Begin by noting that the sign of a partition
$\tau = \{B_{1}, B_{2}, \ldots, B_{n/k}\}$ is
the unique scalar $(-1)^{\tau}$ such that
$t_{B_{1}} \wedge t_{B_{2}} \wedgedots t_{B_{n/k}}
     =
        (-1)^{\tau}
     \cdot
        t_{[n]}$.
Now expand the power in the proposition to obtain that
\begin{eqnarray*}
\left(
\sum_{S}
    f(S) \cdot t_{S}
\right)^{n/k}
  & = &
\sum_{B_{1}}
\cdots
\sum_{B_{n/k}}
f(B_{1}) \cdots f(B_{n/k}) \cdot t_{B_{1}} \cdots t_{B_{n/k}} ,
\end{eqnarray*}
where each sum is over all $k$-element subsets of $[n]$.
Observe that the product in the exterior algebra is
zero if two of the sets have a common element.
Hence the sum reduces to a sum over all ordered
partitions of $[n]$. Ordered here refers to
the set of blocks having a linear order. But given
a partition in $\Pi_{n,k}$ there are $(n/k)!$ ways
to obtain an ordered partition. Hence the sum reduces
to  $(n/k)! \cdot t_{[n]}$ times
the right hand-side of equation~\eqref{equation_Pfaffian},
proving the result.
\end{proof}

\begin{lemma}
Let $f$ be a skew-symmetric function on the set $[n]^{k}$
and let $\sigma$ be a permutation on the set~$[n]$.
Then the function
$g(i_{1}, i_{2}, \ldots, i_{k})
=
 f(\sigma(i_{1}), \sigma(i_{2}), \ldots, \sigma(i_{k}))$
is skew-symmetric and
the two hyperpfaffians differ by the sign $(-1)^{\sigma}$,
that is,
$\Pf(g) = (-1)^{\sigma} \cdot \Pf(f)$.
\label{lemma_switch}
\end{lemma}
\begin{proof}
It is straightforward to observe that $g$ is
skew-symmetric.
It is enough to prove the identity for the adjacent
transposition $\sigma = (j,j+1)$.
Let $u_{i} = t_{\sigma(i)}$, that is, a reordering
of the basis of the exterior algebra.
We claim that
$g(\sigma(S)) \cdot u_{\sigma(S)} = f(S) \cdot t_{S}$.
If neither $j$ nor $j+1$ belong to the set~$S$, there is nothing
to prove. If only one of them belongs to $S$,
yet again, there is nothing to prove. Finally, if
both $j$ and $j+1$ belongs to $S$, we have
that
$g(\sigma(S)) = -f(S)$.
and
$u_{\sigma(S)} = -t_{S}$,
and the two signs cancel.
Hence the two sums
$\sum_{S} g(S) \cdot u_{S}$
and
$\sum_{S} f(S) \cdot t_{S}$
are equal.
Now the result follows from the definition
of the hyperpfaffian and that $u_{[n]} = -t_{[n]}$.
\end{proof}

For more information regarding the hyperpfaffian and its applications, see Redelmeier~\cite{Redelmeier}.

\section{Preliminaries}

A weak composition $\vec{r}$ of an integer $m$ is a vector
$(r_{1}, r_{2}, \ldots, r_{k})$
whose entries are non-negative integers
and their sum is $m$.
The entries are called parts.
For a composition $\vec{r}$ into $k$ parts
we let $x^{\vec{r}}$ denote the monomial
$x_{1}^{r_{1}} x_{2}^{r_{2}} \cdots x_{k}^{r_{k}}$.
Furthermore, let the symmetric group ${\mathfrak S}_{k}$
act on compositions into $k$ parts by reordering the parts.

Let $f(x_{1}, x_{2}, \ldots, x_{k})$ be a homogeneous polynomial of degree
$k/2 \cdot (n-1)$. The polynomial $f$
can be expressed as
$$
  f(x_{1}, x_{2}, \ldots, x_{k})
   =
      \sum_{\vec{r}}
            a_{\vec{r}}
          \cdot
             x^{\vec{r}} ,
$$
where the sum is over all weak compositions $\vec{r}$
of $k/2 \cdot (n-1)$ into $k$ parts.
Furthermore, assume that the polynomial $f$
is skew-symmetric. This implies that
the coefficients satisfy that
$a_{\sigma \circ \vec{r}}=(-1)^{\sigma} \cdot a_{\vec{r}}$.

Let $\Gamma_{n,k}$ denote the set of increasing weak compositions
of $k/2 \cdot (n-1)$ into $k$ distinct parts.
That is, the set $\Gamma_{n,k}$ is given by
$$   \Gamma_{n,k}
     =
       \left\{(r_{1}, r_{2}, \ldots, r_{k}) \in \Nnn^{k}
               \:\: : \:\:
          0 \leq r_{1} < r_{2} < \cdots < r_{k},
          \:
           \sum_{i=1}^{k} r_{i} = k/2 \cdot (n-1) \right\}  .  $$
Hence we can write the skew-symmetric polynomial
$f$ in the form
\begin{equation}
 f(x_{1}, x_{2}, \ldots, x_{k})
   =
      \sum_{\vec{r} \in \Gamma_{n,k}}
      \sum_{\sigma \in {\mathfrak S}_{k}}
             (-1)^{\sigma}
          \cdot
             a_{\vec{r}}
          \cdot
             x^{\sigma \circ \vec{r}}.
 \label{equation_signed_skew_form}
 \end{equation}

We define an oriented partition to be a partition where
each block is endowed with a linear order.
Let~$T_{n,k}$ denote the set of all
oriented partitions $\rho$ of the set $[n]$ where each block has
cardinality $k$.
That is, for an oriented partition
$\rho = \{C_{1}, C_{2}, \ldots, C_{n/k}\}$,
each block $C_{i}$ is a list
$C_{i} = (c_{i,1}, c_{i,2}, \ldots, c_{i,k})$.

Observe that
the number of oriented partitions is given by
$|T_{n,k}| = (k!)^{n/k} \cdot |\Pi_{n,k}| = n!/(n/k)!$.
This can be directly observed by taking a permutation
on $n$ elements and dividing into $n/k$ blocks of size~$k$.
Permuting the $n/k$ blocks, yields the same
oriented partition.
Also observe that since $k$ is even, all the $(n/k)!$
permutations yielding the same oriented partition
have the same sign. We define this sign to be the sign
of the oriented partition, denoted $(-1)^{\rho}$.
More explicitly, the sign of $\rho$
is given by the sign of the permutation

\begin{equation}
 \pi(\rho)  =
 c_{1,1}, \ldots, c_{1,k}, \:\: c_{2,1}, \ldots, c_{2,k}, \:\:
 c_{3,k}, \ldots, c_{n/k,k}
 \label{equation_permutation}
 \end{equation}

By removing the linear order on each block from
the oriented partition $\rho$, we obtain a partition $\tau$.
We note that the sign of the oriented partition $\rho$
and the sign of the partition $\tau$ are related by
\begin{equation}  (-1)^{\rho}
    =
       (-1)^{\tau}
     \cdot
       (-1)^{\sigma_{1}}
     \cdot
       (-1)^{\sigma_{2}}
     \cdots
       (-1)^{\sigma_{n/k}}   ,
\label{equation_perm_sign}
\end{equation}
where $\sigma_{i}$ is the permutation
on the set $\{c_{i,1}, c_{i,2}, \ldots, c_{i,k}\}$
that orders the $i$th block, that is,
$\sigma_{i}(c_{i,1}) < \sigma_{i}(c_{i,2}) < \cdots < \sigma_{i}(c_{i,k})$.

Let $R_{n,k}$ denote the collection of sets of
size $n/k$ of
compositions in $\Gamma_{n,k}$,
where all the parts of the compositions are distinct.
Let $\beta=\{\vec{r}_{1},\ldots, \vec{r}_{n/k}\}$ denote such a set in $R_{n,k}$.
Observe that the sum of all the entries
of the compositions is given by
$n/k \cdot k/2 \cdot (n-1)$ which is the sum
$0 + 1 + \cdots + (n-1)$. Hence we conclude that
the underlying parts of the compositions of $\beta$
are the integers $0$ through $n-1$.
Thus we view $\beta$ as
an oriented set partition
of the elements $\{0,\ldots , n-1\}$ into $n/k$ blocks of size~$k$
in which each block is a composition in $\Gamma_{n,k}$.
Define the sign of
$\beta=\{\vec{r}_{1},\ldots, \vec{r}_{n/k}\}\in R_{n,k}$ with
$\vec{r}_{i}=(r_{i,1}, \ldots, r_{i,k})$, denoted by $(-1)^{\beta}$,
to be the sign of the permutation
$$ \pi(\beta)
    =
    r_{1,1}, \ldots, r_{1,k}, \:\: r_{2,1}, \ldots, r_{2,k},
      \:\:  r_{3,k}, \ldots, r_{n/k,k},  $$
where $\pi(\beta)$ is a permutation of the elements $\{0,1,\ldots, n-1\}$.

\section{Main Theorem}

Using the skew-symmetric polynomial given in
equation~\eqref{equation_signed_skew_form},
we have the following identity.
\begin{theorem}
The hyperpfaffian $\Pf(f(x_S))$ of order n is the product of the Vandermonde
product with a signed sum of products of coefficients $a_{\vec{r}}$:
$$\Pf(f(x_S))_{S\in \binom{[n]}{k}}
		=\left(\sum_\beta
			(-1)^{\beta} \cdot \prod_{i=1}^{n/k} a_{\vec{r}_{i}}\right)
                     \cdot
			 \prod_{1 \leq i < j \leq n}
                   	(x_{j} - x_{i}), $$
where the sum ranges over all partitions
$\beta$ in $R_{n,k}$.
\label{main_theorem}
\end{theorem}

\begin{example}
{\rm
When $n=12$ and $k=4$ there are
there are $32$ oriented partitions in $R_{12,4}$.
The coefficient in Theorem~\ref{main_theorem}
is in this case given by
$$ \begin{array}{l}
\:\:\:\:
   a_{0,1,10,11} a_{2,3,8,9} a_{4,5,6,7}
+ a_{0, 1, 10, 11} a_{2, 4, 7, 9} a_{3, 5, 6, 8}
+ a_{0, 1, 10, 11} a_{2, 5, 6, 9} a_{3, 4, 7, 8}
+ a_{0, 1, 10, 11} a_{2, 5, 7, 8} a_{3, 4, 6, 9} \\
+ a_{0, 2, 9, 11} a_{1, 3, 8, 10} a_{4, 5, 6, 7}
+ a_{0, 2, 9, 11} a_{1, 4, 7, 10} a_{3, 5, 6, 8}
+ a_{0, 2, 9, 11} a_{1, 5, 6, 10} a_{3, 4, 7, 8}
- a_{0, 2, 9, 11} a_{1, 6, 7, 8} a_{3, 4, 5, 10} \\
+ a_{0, 3, 8, 11} a_{1, 2, 9, 10} a_{4, 5, 6, 7}
+ a_{0, 3, 8, 11} a_{1, 4, 7, 10} a_{2, 5, 6, 9}
+ a_{0, 3, 8, 11} a_{1, 5, 6, 10} a_{2, 4, 7, 9}
+ a_{0, 3, 8, 11} a_{1, 5, 7, 9} a_{2, 4, 6, 10} \\
+ a_{0, 3, 9, 10} a_{1, 2, 8, 11} a_{4, 5, 6, 7}
- a_{0, 3, 9, 10} a_{1, 4, 6, 11} a_{2, 5, 7, 8}
- a_{0, 3, 9, 10} a_{1, 6, 7, 8} a_{2, 4, 5, 11}
+ a_{0, 4, 7, 11} a_{1, 2, 9, 10} a_{3, 5, 6, 8} \\
+ a_{0, 4, 7, 11} a_{1, 3, 8, 10} a_{2, 5, 6, 9}
+ a_{0, 4, 7, 11} a_{1, 5, 6, 10} a_{2, 3, 8, 9}
+ a_{0, 4, 8, 10} a_{1, 3, 7, 11} a_{2, 5, 6, 9}
+ a_{0, 4, 8, 10} a_{1, 5, 7, 9} a_{2, 3, 6, 11} \\
+ a_{0, 5, 6, 11} a_{1, 2, 9, 10} a_{3, 4, 7, 8}
+ a_{0, 5, 6, 11} a_{1, 3, 8, 10} a_{2, 4, 7, 9}
+ a_{0, 5, 6, 11} a_{1, 4, 7, 10} a_{2, 3, 8, 9}
+ a_{0, 5, 6, 11} a_{1, 4, 8, 9} a_{2, 3, 7, 10} \\
- a_{0, 5, 7, 10} a_{1, 2, 8, 11} a_{3, 4, 6, 9}
+ a_{0, 5, 7, 10} a_{1, 4, 6, 11} a_{2, 3, 8, 9}
+ a_{0, 5, 7, 10} a_{1, 4, 8, 9} a_{2, 3, 6, 11}
+ a_{0, 5, 8, 9} a_{1, 3, 7, 11} a_{2, 4, 6, 10} \\
+ a_{0, 5, 8, 9} a_{1, 4, 6, 11} a_{2, 3, 7, 10}
+ a_{0, 5, 8, 9} a_{1, 4, 7, 10} a_{2, 3, 6, 11}
- a_{0, 6, 7, 9} a_{1, 2, 8, 11} a_{3, 4, 5, 10}
- a_{0, 6, 7, 9} a_{1, 3, 8, 10} a_{2, 4, 5, 11} .
\end{array} $$
}
\end{example}

Let $W_{n,k}$ be the set of oriented partitions
$\rho=\{C_{1}, C_{2},\ldots ,C_{n/k}\}$ on the set $[n]$ with a composition
$\vec{w}_{i}=(w(c_{i,1}),\ldots, w(c_{i,k}))\in \Gamma_{n,k}$
assigned to each block $C_{i}=(c_{i,1},\ldots ,c_{i,k})$.
We define the following notions for such a weighted oriented partition $\rho$.
Let $(-1)^\rho$ be the {\em sign} of $\rho$, defined as in
the previous section by the sign of the permutation $\pi(\rho)$ from
equation~\eqref{equation_permutation}.
Let the {\em coefficient}~$c(\rho)$ denote the product
$\prod_{i=1}^{n/k}a_{\vec{w}_{i}}$ determined by the weight vectors of $\rho$.
Lastly, let $w(\rho)$ denote the {\em monomial}
$\prod_{i=1}^{n/k} x_{C_{i}}^{\vec{w}_{i}}$ where
$x_{C_{i}}^{\vec{w}_{i}}=\prod_{j=1}^{k} x_{c_{i,j}}^{w(c_{i,j})}$.

\begin{lemma}
The following expansion holds for the hyperpfaffian:
$$\Pf\left(f(x_S)\right)_{S\in \binom{[n]}{k}}
=
\sum_{\rho \in {W_{n,k}}} (-1)^{\rho} \cdot c(\rho) \cdot w(\rho)  . $$
\label{lemma_expanded_pfaffian}
\end{lemma}

\begin{proof}

Expanding equation~\eqref{equation_Pfaffian} and applying
equation~\eqref{equation_signed_skew_form}, we have
$$ \Pf\left(f(x_S)\right)_{S \in \binom{[n]}{k}}
=
   \sum_{\tau \in \Pi_{n,k}}
                  (-1)^{\tau}
                \cdot
                   \prod_{i=1}^{n/k}  \left(\sum_{\vec{r} \in \Gamma_{n,k}}
      \sum_{\sigma \in {\mathfrak S}_{k}}
             (-1)^{\sigma}
          \cdot
             a_{\vec{r}}
          \cdot
             x_{B_{i}}^{\sigma \circ \vec{r}}\right).$$
Using the distributive law, expand the above product.  We obtain an oriented, weighted partition~$\rho$ for each term by orienting the elements in each block $B_{i}\in \tau$ by increasing size of the exponents of their associated variables.  The composition $\vec{r}$ corresponds to the choice of weight vector for each block, and the permutation $\sigma$ will undo the orientation of the block to properly assign the weights as exponents.  Multiplying the sign of $\sigma$ for each block with the sign of $\tau$ gives the sign of $\rho$ as described in equation~\eqref{equation_perm_sign} because for block $B_{i}$, $\sigma=\sigma_{i}^{-1}$.

\end{proof}

Let $W_{n,k}^{r}$ denote the subset of $W_{n,k}$ with repeated weights, and let $W_{n,k}^{d}$ denote the complement, that is, partitions with distinct weights.  We now create a sign-reversing involution $\phi$ for the set $W_{n,k}^{r}$ to narrow our focus to only partitions with distinct weights.  Given a partition $\rho$ in $W_{n,k}^{r}$, let $(i,j)$ be the lexicographically smallest pair of elements in $[n]$ in which $w(i)=w(j)$.  Define $\phi(\rho)$ by swapping $i$ and $j$, while leaving the weight vector for each block and the orientation unchanged.

\begin{lemma}
The function $\phi$ is a sign-reversing involution on the set $W_{n,k}^{r}$ which does not change the coefficient nor the monomial.  That is, for an oriented partition $\rho$ we have that $\phi^2(\rho)=\rho$, $c(\phi(\rho))=c(\rho)$, $w(\phi(\rho))=w(\rho)$, but $(-1)^{\phi(\rho)}=-(-1)^{\rho}$.
\label{lemma_sign_reversing}
\end{lemma}
\begin{proof}
By definition, it follows that $\phi$ is an involution, and that it
leaves the coefficient and the monomial unchanged. 
To see that $\phi$ is sign-reversing, consider the consequences
of swapping $i$ and $j$ within the permutation $\pi(\rho)$.
We get that $\pi(\phi(\rho))=(i, j)\circ \pi(\rho)$,
hence the transposition changes the sign of
the corresponding permutation as $\phi$ is applied.
Thus, $(-1)^{\phi(\rho)}=-(-1)^{\rho}$.
\end{proof}

\begin{figure}
\begin{center}
\setlength{\unitlength}{1mm}
\begin{picture}(110,43)(0,0)

\thicklines

\multiput(0,6)(10,0){12}{\circle*{1.5}}

\put(-5,0){\makebox(10,5)[c]{$1$}}
\put(5,0){\makebox(10,5)[c]{$2$}}
\put(15,0){\makebox(10,5)[c]{$3$}}
\put(25,0){\makebox(10,5)[c]{$4$}}
\put(35,0){\makebox(10,5)[c]{$5$}}
\put(45,0){\makebox(10,5)[c]{$6$}}
\put(55,0){\makebox(10,5)[c]{$7$}}
\put(65,0){\makebox(10,5)[c]{$8$}}
\put(75,0){\makebox(10,5)[c]{$9$}}
\put(85,0){\makebox(10,5)[c]{$10$}}
\put(95,0){\makebox(10,5)[c]{$11$}}
\put(105,0){\makebox(10,5)[c]{$12$}}

\qbezier(80,6)(40,70)(0,6)
\qbezier(0,6)(5,14)(10,6)
\qbezier(10,6)(20,20)(29,7)
\put(29,7){\line(0,1){3}}
\put(29,7){\line(-1,0){3}}

\qbezier(40,6)(30,20)(20,6)
\qbezier(20,6)(45,50)(70,6)
\qbezier(70,6)(80,20)(89,7)
\put(89,7){\line(0,1){3}}
\put(89,7){\line(-1,0){3}}

\qbezier(100,6)(105,14)(110,6)
\qbezier(110,6)(85,50)(60,6)
\qbezier(60,6)(55,14)(51,7)
\put(51,7){\line(0,1){3}}
\put(51,7){\line(1,0){3}}

\put(30,38){\makebox(20,5)[c]{$(1,4,5,12)$}}
\put(35,28){\makebox(20,5)[c]{$(0,1,7,14)$}}
\put(75,28){\makebox(20,5)[c]{$(2,4,6,10)$}}

\end{picture}
\end{center}
\caption{
The oriented partition
$\{(9,1,2,4), (5,3,8,10), (11,12,7,6)\}$.
Note that it has a negative sign.
The weights yield the monomial
$x_{9}^{1} x_{1}^{4} x_{2}^{5} x_{4}^{12} \cdot
x_{5}^{0} x_{3}^{1} x_{8}^{7} x_{10}^{14} \cdot
x_{11}^{2} x_{12}^{4} x_{7}^{6} x_{6}^{10}$
and the coefficient
$a_{1,4,5,12} \cdot a_{0,1,7,14} \cdot a_{2,4,6,10}$.}
\label{figure_o_p}
\end{figure}
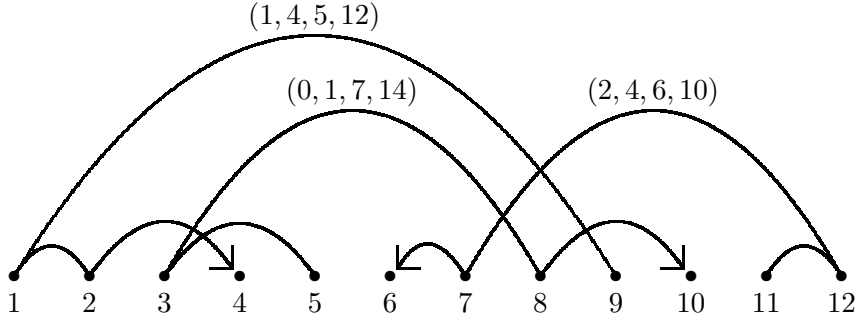

Observe that the weighted oriented partition in
Figure~\ref{figure_o_p} has non-distinct powers
of the pair of variables $x_{1}$ and $x_{12}$;
and the pair $x_{3}$ and $x_{9}$, and that the pair $(1,12)$
is lexicographic least. Hence this weighted oriented
partition cancels with the oriented partition
$\{(9,12,2,4), (5,3,8,10), (11,1,7,6)\}$.

We now concentrate on weighted oriented partitions where
the weights are distinct, that is, the set~$W_{n,k}^{d}$.
Note that this implies that
the weights are $0$ through $n-1$
and
allows us to narrow our focus to weight vectors that make up
an oriented partition in $R_{n,k}$.

For a weighted oriented partition $\rho$
in~$W_{n,k}^{d}$ let $\sigma$
be the unique permutation such that $w(\rho)=
x_{1}^{\sigma_{1}-1} x_{2}^{\sigma_{2}-1}\cdots x_{n}^{\sigma_{n}-1}$.
Furthermore, let
$\beta$ be the set of weight vectors
assigned to the blocks of $\rho$,
that is, $\beta$ lies in $R_{n,k}$.
Observe that this describes a bijection
between $W_{n,k}^{d}$
and the Cartesian product of
the symmetric group ${\mathfrak S}_{n}$
and the weight vectors $R_{n,k}$.

\begin{lemma}
The sign of a weighted oriented partition $\rho$ in $W_{n,k}^{d}$
factors as $(-1)^{\rho} = (-1)^{\beta} \cdot (-1)^{\sigma}$.
\label{lemma_sign}
\end{lemma}
\begin{proof}
Define the permutation $\pi(\beta)'$ on $[n]$ such that $\pi(\beta)'_i=\pi(\beta)_i+1$.  Since $(-1)^{\pi(\beta)'}=(-1)^{\pi(\beta)}$, it is enough to observe that the permutation $\pi(\beta)'$ factors as $\sigma \circ \pi(\rho)$.
\end{proof}

\begin{proof}[Proof of Theorem~\ref{main_theorem}.]
By combining
Lemmas~\ref{lemma_expanded_pfaffian}
through~\ref{lemma_sign}, we have that
$$\Pf\left(f(x_S)\right)_{S\in \binom{[n]}{k}}
=
\sum_{\rho \in {W_{n,k}^{d}}} (-1)^{\rho} \cdot c(\rho) \cdot w(\rho)
=
\left(\sum_\beta
			(-1)^{\beta} \cdot \prod_{i=1}^{n/k} a_{\vec{r}_{i}}\right)
\cdot \sum_{\sigma\in {\mathfrak S}_{n}}
       (-1)^{\sigma} \cdot
       x_{1}^{\sigma_{1} - 1} \cdots x_{n}^{\sigma_{n} - 1}, $$
where the last sum is the Vandermonde determinant,
which is equal to the Vandermonde product.
\end{proof}

An algebraic proof of Theorem~\ref{main_theorem}
 is as follows. Note that when setting two of the variables
$x_{i}$ and $x_{j}$ equal, the hyperpfaffian vanishes
by Lemma~\ref{lemma_switch}.
Hence as a polynomial in $x_{1}$ through~$x_{n}$,
the Vandermonde product divides the hyperpfaffian.
However, the two sides have same degree
$n/k \cdot k/2 \cdot (n-1) = \binom{n}{2}$, and hence
are equal up to a constant.
By considering the coefficient of the term
$x_{2} \cdots x_{n-1}^{n-2} x_{n}^{n-1}$,
we obtain the constant
$\sum_{\beta} (-1)^{\beta} \cdot \prod_{i=1}^{n/k} a_{\vec{r}_{i}}$.

Finally, observe that when the polynomial $f$ is replaced with
a polynomial of degree less than $k/2 \cdot (n-1)$ the
hyperpfaffian will be zero. This can be seen in two ways.
Either, the only polynomial of degree less than $\binom{n}{2}$
which is divisible by the Vandermonde product is the zero
polynomial. Alternatively, the sign reversing
involution has no fixed points; that is, it cancels all the terms.

\section{Application to the classical Pfaffian}

Let us now focus on the $k=2$ case.
In this case, the oriented partitions devolve into directed matchings,
and the compositions in $\Gamma_{n,2}$ have two parts with
the sum $n-1$, hence they have
the form $(i,n-1-i)$ from $i=0,1,\ldots, n/2-1$.
This leads the skew polynomial $f$ to have the following form,
in which we abbreviate the coefficients $a_{i,n-1-i}$ as simply $a_{i}$,
$$ f(x,y) = \sum_{i=0}^{n-1} a_{i} \cdot x^{i} y^{n-1-i}, $$
where $a_{n-1-i} = -a_{i}$.  Since the only
oriented partition in $R_{n,2}$ is
$\{(0,n-1),(1,n-2),\ldots, (n/2-1,n/2)\}$,
which has the sign $(-1)^{2\binom{n/2}{2}}=1$,
Theorem~\ref{main_theorem}
reduces to the following corollary.

\begin{corollary}
The Pfaffian $\Pf(f(x_{i},x_{j}))$ of order $n$
is the product of the first $n/2$ of the coefficients~$a_{i}$
times the Vandermonde product:
$$
\Pf\left(f(x_{i}, x_{j})\right)_{1 \leq i < j \leq n}
    =
\prod_{i=0}^{n/2-1} a_{i}
\cdot
        \prod_{1 \leq i < j \leq n}
                   (x_{j} - x_{i})    .    $$
\label{corollary}
\end{corollary}

As a corollary we have the following identity due to
Torelli~\cite{Torelli},
see also~\cite[Equation~(4.6)]{Knuth}.
\begin{corollary}[Torelli]
When the skew-symmetric polynomial
is the function $f(x,y) = (y-x)^{n-1}$, the Pfaffian
is given by
$$
\Pf\left(f(x_{i} , x_{j})\right)_{1 \leq i < j \leq n}
    =
(-1)^{\binom{n/2}{2}}  
\cdot
\prod_{i=0}^{n/2 - 1} \binom{n-1}{i}
\cdot
        \prod_{1 \leq i < j \leq n}
                   (x_{j} - x_{i})    .    $$
\label{corollary_Torelli}
\end{corollary}
It is enough to observe that
$a_{i} = (-1)^{i} \cdot \binom{n-1}{i}$.

\section{Concluding remarks}

Benjamin and Dresden~\cite{Benjamin_Dresden}
gave a different combinatorial proof of the Vandermonde
identity than Gessel~\cite{Gessel}.
They gave a combinatorial interpretation of
both sides and then used a sign reversing involution
on the opposite side than Gessel. Is it possible to
prove Corollary~\ref{corollary}
or more generally, Theorem~\ref{main_theorem}
by a similar technique?

Which other identities does the hyperpfaffian satisfy?
See
Knuth~\cite{Knuth}
and 
Tanner~\cite{Lloyd_Tanner}
for the expansion for
products of two overlapping Pfaffians, and for applications
of this identity. Can any of these results be generalized for
hyperpfaffians? One such example is
the following identity for compositions of the hyperpfaffians, 
proved by Luque and Thibon~\cite{Luque_Thibon}.
\begin{theorem}
Let $k$, $n$ and $p$ be three even positive integers
such that $n$ is a multiple of $k$ and $p$ is a multiple of $n$.
Let $f$ be a skew-symmetric $k$-ary function on the set $[p]$.
Define an $n$-ary function~$g$ by the hyperpfaffian of order $n$,
that is,
$$ g(i_{i}, \ldots, i_{n})
   =
     \Pf(f)_{(i_{1}, \ldots, i_{n})}  .  $$
Then the hyperpfaffian of order $p$ of the function $g$ is given by
a constant times the hyperpfaffian of $f$ of order~$p$, that is,
$$   \Pf(g)
    =
       \frac{1}{(p/n)!} \cdot \binom{p/k}{n/k, \ldots, n/k} \cdot \Pf(f)   ,  $$
where there are $p/n$ instances of $n/k$ in the multinomial coefficient.
\end{theorem}

\section*{Acknowledgments}

The authors are grateful to the referees for
their comments on an earlier version of this note.
The first author was partially supported by
National Science Foundation grant DMS~0902063.

\newcommand{\journal}[6]{{#1,} #2, {\it #3} {\bf #4} (#5) #6.}
\newcommand{\thesis}[4]{{#1,} #2, {\it #3,} (#4).}

\bigskip

{\em R.\ Ehrenborg, N.\ B.\ Fox.
Department of Mathematics,
University of Kentucky,
Lexington, KY 40506-0027,}
{\tt jrge@ms.uky.edu},
{\tt norman.fox@uky.edu}

\end{document}